\documentclass[oneside,english]{amsart}

\usepackage{geometry}
\usepackage{amsmath}
\usepackage{amsthm} % \begin{proof} \end{proof}
\usepackage{amssymb}  % para fazer IR, IN, etc..
\usepackage{epsfig}

\usepackage[mathcal]{eucal}
\usepackage{subeqnarray}

\usepackage{url}

\newcommand{\R}{\mathbb{R}}

\newcommand{\inner}[2]{\langle{#1},{#2}\rangle}

\newcommand{\tos}{\rightrightarrows} % point-to-set mappings

\newtheorem{theorem}{Theorem}[section]
\newtheorem{lemma}[theorem]{Lemma}
\newtheorem{corollary}[theorem]{Corollary}
\newtheorem{remark}[theorem]{Remark}
\newtheorem{proposition}[theorem]{Proposition}

\usepackage{color}

\begin{document}

\title{Stability of a Regularized Newton Method with two Potentials }

\author{Boushra ABBAS\\
\\
ACSIOM, I3M UMR CNRS 5149, Universit\'e Montpellier 2}

\date { May 20, 2015}

\begin{abstract}
  \noindent In a Hilbert space setting, we study the stability properties of  the regularized continuous Newton method with two potentials,  which aims at solving inclusions governed by structured monotone operators. The  Levenberg-Marquardt regularization term acts in an open loop way. As a byproduct of our study, we can take the regularization coefficient of bounded variation. These stability results are directly related to the study of numerical algorithms that combine forward-backward and Newton's methods.

 \end{abstract}
 
 \maketitle
 
 \bigskip

\pagestyle{plain}

%%%%%%%%%%%%%%%%%%%%%%%%%%%%%%%%%%%%%%%%%%%%%%%%%%%%%%%%%%%%%%%%%%%%%%%%%%%%%%%%%%%%%%%%

\section{{\large Introduction }}

Throughout this paper $H$ is  a real Hilbert space with  scalar product $\left\langle \cdot,\cdot\right\rangle $   
and norm $\left\Vert x\right\Vert =\sqrt{\left\langle x,x\right\rangle }$
for $x\in H$.  We make the following standing assumptions:  $\varphi$ and $\psi$ are two functions which act on $H$ and satisfy

\smallskip

$\bullet$  $\varphi:H\longrightarrow\mathbb{R}\cup\left\{ +\infty\right\} $ is  convex lower semicontinuous, and proper;

$\bullet$ $\psi:H\longrightarrow\mathbb{R}$ is  convex differentiable, and  $\nabla\psi$ is
 Lipschitz continuous on the bounded subsets of $H$.
 
$\bullet$  $\varphi+\psi$ is bounded from below on $H.$
 
 \smallskip

\noindent We are concerned with the study of Newton-like continuous and discrete dynamics attached to solving the  the structured minimization problem
$$
(\mathcal P ) \quad  \min \left\{  \varphi (x) + \psi (x): \quad x\in H   \right\}.
$$

Note the asymmetry between $\varphi$, which may be nonsmooth, with extended real values, and $\psi$ which is continuously differentiable, whence the structured property of the above problem. Indeed, we wish to design continuous and discrete dynamics which exploit this particular structure and involve $\varphi$ via implicit operations (like resolvent or proximal operators) and $\psi$ via explicit operations (typically gradient-like methods).
So doing we expect obtening new forward-backward splitting methods involving Riemannian metric aspects, and which are close to the Newton method.
This approach has been delineated in a series of recent papers, \cite{AA}, \cite{AAS}, \cite{ARS}, \cite{AS}.
In this paper we are concerned with the stability properties with respect to the data ($\lambda$, $x_{0}$, $\upsilon_{0}$, ...) of the strong solutions of the differential
inclusion 
\begin{subeqnarray} \label{basic}
& & v(t)\in \partial \varphi(x(t)) \slabel{basica} \\
& & \lambda(t)\dot x(t)+\dot v(t)+v(t)+ \nabla \psi(x(t))=0  \slabel{basicb} \\
& & x\left(0\right)=x_{0},\upsilon\left(0\right)=\upsilon_{0}.
\slabel{basicc} 
\end{subeqnarray}
Let us now make our standing assumption on  function $\lambda(\cdot)$: 
\begin{align}
\label{ac-lambda}
&   \lambda: [0,+\infty[ \longrightarrow \  ]0,+\infty [  \  \mbox{is   continuous,}\\
 & \mbox{ and absolutely continuous on each  interval} \ [0,b],  \   0<b< +\infty.
\end{align}
Hence $\dot{\lambda}(t)$ exists for almost every $t>0$, and $\dot{\lambda}(\cdot)$ is Lebesgue integrable on each bounded interval $\left[0,b\right]$.
 We stress the fact that we assume  $\lambda (t) > 0$, \  for any $t \geq 0$. By continuity of $\lambda(\cdot)$, this implies that, for any 
$b>0$, there exists some positive finite lower and upper bounds for $\lambda(\cdot)$ on $\left[0, b \right]$, i.e.,  for any $t \in \left[0, b \right]$
\begin{equation}
\label{localbound}
0 < {\lambda}_{b, min} \leq \lambda (t) \leq {\lambda}_{b, max} < + \infty.
\end{equation}
 Our main  interest is to allow $\lambda(t)$ to go to zero as 
$t\to +\infty$. This makes the corresponding Levenberg-Marquardt regularization method asymptotically close to the Newton's method.\\
Let us summarize the results obtained in  \cite{AAS}. Under the above assumptions, for any Cauchy data  $ x_{0} \in \mbox{dom} \partial \varphi$ and $\upsilon_{0} \in \partial  \varphi (x_{0})$,
there exists a unique strong global solution $\left(x\left(\cdot\right),\upsilon\left(\cdot\right)\right):[0,+\infty [ \rightarrow H\times H$
of the Cauchy problem (\ref{basica})-(\ref{basicc}).  Assuming that
the solution set is nonempty, 
 if $\lambda\left(t\right)$ tends to zero not too fast,
as $t\longrightarrow+\infty$, then $\upsilon\left(t\right)\longrightarrow0$
strongly, and $x\left(t\right)$ converges weakly to some equilibrium
which is a solution of the  minimization problem $(\mathcal P )$.
By Minty representation of $\partial\varphi$,  the solution pair $\left(x\left(\cdot\right),\upsilon\left(\cdot\right)\right)$ of
(\ref{basica})-(\ref{basicc}) can be represented as
follows: set $\mu(t) = \frac{1}{\lambda (t)}$, then for any $t\in[0,+\infty)$, 
\begin{eqnarray}
 &  & x\left(t\right)=\textrm{prox}_{\mu\left(t\right)\varphi}\left(z\left(t\right)\right);\\
 &  & \upsilon\left(t\right)=\nabla \varphi_{\mu\left(t\right)}\left(z\left(t\right)\right),
\end{eqnarray}
where 
  $z\left(\cdot\right):[0,+\infty [\rightarrow H$
is the unique strong global solution of the Cauchy problem 
\begin{eqnarray} 
& & \dot{z}(t)+(\mu(t)-\dot{\mu}(t))\nabla\varphi_{\mu(t)}(z(t))+\mu(t)\nabla\psi\left(\textrm{prox}_{\mu\left(t\right)\varPhi}\left(z\left(t\right)\right)\right)=0 \label{basic.z1} \\
& & z\left(0\right)=x_{0}+\mu\left(0\right)\upsilon_{0}. \label{basic.z2} 
\end{eqnarray}

Let us recall that $\textrm{prox}_{\mu\varphi}$ is the proximal mapping associated to $\mu\varphi$. Equivalently, $\textrm{prox}_{\mu\varphi}=\left(I+\mu\partial\phi\right)^{-1}$ is the resolvent of index $\mu>0$ of the maximal monotone operator $\partial\varphi$, and $\nabla\varphi_{\mu}$ is its Yosida approximation of index $\mu >0$.

Let us stress the fact
that, for each $t>0$, the operators $\textrm{prox}_{\mu(t)\varphi}:H\longrightarrow H$,
$\nabla\varphi_{\mu(t)}:H\longrightarrow H$ are everywhere
defined and Lipschitz continuous, which makes this system relevant
to the Cauchy\textendash Lipschitz theorem in the nonautonomous case, which naturally suggests good stability results of the solution of (\ref{basica})-(\ref{basicc}) with respect to the data. 

$\vphantom{}$

This paper is organized as follows: We first establish a priori  energy estimates  on the trajectories. Then we consider the case where
$\lambda$ is locally absolutely continuous. Note that it is important, for numerical reasons, to study the stability
of the solution with respect to perturbations of the data, and in particular of $\lambda$ which plays a crucial role in the regularization process. In Theorem
\ref{basic-stab} we prove the Lipschitz continuous dependence of the solution with
respect to $\lambda$. Moreover, the Lipschitz constant only depends
on the $L^{1}$ norm of the time derivative of $\lambda$. Finally, we
extend our analysis to the case where $\lambda$ is a function with
bounded variation (possibly involving jumps). We use a regularization
by convolution method in order to reduce to the smooth case, and then
pass to the limit in the equations. So doing, in Theorem \ref{BV-Theorem} and Corollary
\ref{BV-corollary}, we prove the existence and uniqueness of a strong solution for
(1a)-(1c), in the case where $\lambda$ is a function with
bounded variation.

\section{{\large A priori estimates}}

 The linear space $H\times H$ is equipped with its usual
Hilbertian norm $\left\Vert \left(\xi,\zeta\right)\right\Vert =\sqrt{\left\Vert \xi\right\Vert ^{2}+\left\Vert \zeta\right\Vert ^{2}.}$
In this section we work on a fixed bounded interval $\left[0, T \right]$, and following assumption (\ref{ac-lambda}), we suppose that there exists some positive constant $c_0$ such that
\begin{equation}
\label{localbound2}
0 < c_0 \leq \lambda (t)  \quad \mbox{for all} \ t\in [0,T].
\end{equation}
We will also assume that
$\nabla \psi$ is $L_{\psi}$-Lipschitz continuous. Indeed, this is not a restrictive assumption since one can reduce the study to trajectories belonging to a fixed ball in $H$.
We will often omit the time variable $t$ and write $x,$$\upsilon$....
for $x\left(t\right)$, $\upsilon\left(t\right)$.... when no ambiguity
arises.

In the following two Propositions we denote by  $\left(x\left(\cdot\right),\upsilon\left(\cdot\right)\right):[0,+\infty)\rightarrow H\times H$ the unique strong global solution
of the Cauchy problem (\ref{basica})-(\ref{basicc}).

\begin{proposition} Let $\left(x,\upsilon\right)$ be   the  strong solution of system (\ref{basica})-(\ref{basicc}) on $\left[0,T\right]$, $T>0$. Then 
  \label{pr:energy1}
\begin{align}
&\int_0^T \| \dot x (t) \|^2   \leq \frac{1}{c_0}\left(  ( \varphi + \psi)(x_0) - \inf_H ( \varphi + \psi)            \right).\\
&\|x\|_{L^{\infty}(0,T; H)}   \leq   \|x_0\| +   \sqrt{ \frac{T}{c_0}}\left(  ( \varphi + \psi)(x_0) - \inf_H ( \varphi + \psi)            \right)^{\frac{1}{2}}.
\end{align}
\end{proposition}
\begin{proof}
a)  For almost   every $t>0,$  \ $\dot x (t)$ and $\dot {v}(t)$ are well defined, thus
  \[
  \inner{\dot x(t)}{\dot v(t)}=\lim_{h\to 0} \ 
  \frac{1}{h^2}\inner{x(t+h)-x(t)}{v(t+h)-v(t)}.
  \]
  By equation (\ref{basica}), we have  $v(t)\in \partial \varphi(x(t))$. Since $\partial \varphi: H \tos H$ is monotone
  \[
\inner{x(t+h)-x(t)}{v(t+h)-v(t)} \geq 0.
  \]
 Dividing by $h^2$ and passing to the limit preserves the inequality, which   yields  
 \begin{equation} \label{energy2}
 \inner{\dot x (t)}{\dot v(t) }\geq 0.
 \end{equation}
By taking
  the inner product of both sides of (\ref{basicb}) by $\dot{x}(t)$ we obtain
  \[
 \lambda(t)\|\dot x(t)\|^2 + \inner{\dot x (t)}{\dot v(t) }   + \inner{ v(t) } {\dot x (t)}  + \inner{ \nabla \psi(x(t))}{\dot x (t)}=0.
  \]
 By  (\ref{energy2}) we infer 
 \begin{equation} \label{energy3}
 \lambda(t)\|\dot x(t)\|^2   + \inner{ v(t) } {\dot x (t)}  + \inner{ \nabla \psi(x(t))}{\dot x (t)}\leq 0.
  \end{equation}
	Noticing that $x$ and $v$ are continuous on $\left[0,T\right]$, hence bounded, and that $\lambda$ is bounded from below on $\left[0,T\right]$ by a positive number,
one easily gets from (\ref{energy3}) that 
\begin{equation}\label{energy0} 
\dot{x} \in L^2 (0,T; H).
\end{equation}
For our stability analysis, we now establish a precise estimate of the $L^2$ norm of $\dot{x}$.
By the classical derivation chain rule
 \begin{equation} \label{energy4}
  \frac{d}{dt}  \psi(x(t)) =                \inner{ \nabla \psi(x(t))}{\dot x (t)}.
   \end{equation}
	We appeal to  a similar formula which is still valid for a convex lower semicontinuous proper function  $\varphi: H \to \R \cup \left\{+\infty\right\}$.
Notice  that

$i)$ $v(t)\in \partial \varphi(x(t))$ for almost every $t\in [0,T]$;
  
$ii)$ $v$ is continuous on $\left[0,T\right]$, and hence belongs to $L^2 (0,T; H)$;

$iii)$ $\dot{x} \in L^2 (0,T; H)$  by (\ref{energy0}).

\noindent By $i)$,  $ii)$, $iii)$, conditions of  \cite[Lemma 3.3]{Br} are satisfied, which allows to deduce that $t \mapsto \varphi\left(x\left(t\right)\right)$ is absolutely continuous 
on $\left[0,T\right]$, and for almost $t\in\left[0,T\right]$,
 \begin{equation} \label{energy5}
  \frac{d}{dt}  \varphi(x(t)) =                \inner{ v(t)}{\dot x (t)}.
  \end{equation}
 Combining  (\ref{energy3}) with   (\ref{energy4})  and (\ref{energy5}) we obtain
  \begin{equation} \label{energy6}
 \lambda(t)\|\dot x(t)\|^2  +   \frac{d}{dt} \left( \varphi(x(t))  +  \psi(x(t))     \right) \leq 0.
  \end{equation}
 By integrating the above inequality from $0$ to $T$ we obtain
  \begin{equation} \label{energy7}
 \int_0^T  \lambda(t)\|\dot x(t)\|^2 dt  +    (\varphi  +  \psi)(x(T))     \leq (\varphi  +   \psi)(x(0)).
  \end{equation}
Since $\varphi+\psi$ is bounded from below on $H$, and $\lambda$ is minorized  by the  positive constant $c_0$ on $[0,T]$ (see (\ref{localbound2})), we infer
  \begin{equation} \label{energy8}
 \int_0^T \| \dot x (t) \|^2 dt  \leq \frac{1}{c_0}\left(  ( \varphi + \psi)(x_0) - \inf_H ( \varphi + \psi)            \right).
  \end{equation}
 
 b) Since $x(\cdot)$ is absolutely continuous on bounded sets, for any $t\in [0,T]$
 \begin{equation} \label{sup1}
 x(t) =   x_0 +   \int_0^t  \dot x (\tau) d \tau.
  \end{equation}
 Passing to the norm, and using Cauchy-Schwarz inequality yields
 \begin{align} \label{sup2}
 \|x(t)\| &\leq  \| x_0 \| +   \int_0^t  \| \dot x (\tau)\| d \tau \\ 
    & \leq  \| x_0 \| +  \sqrt{T} \left(\int_0^T  \| \dot x (\tau)\|^2 d \tau \right)^{\frac{1}{2}} .
  \end{align}
Combining the above inequality with (\ref{energy8}) gives
 \begin{equation} \label{sup3}
 \|x(t)\| \leq    \| x_0 \| +  \sqrt{\frac{T}{c_0}} \left(  ( \varphi + \psi)(x_0) - \inf_H ( \varphi + \psi)   \right)^{\frac{1}{2}}  .   
  \end{equation}
This being true for any  $t\in [0,T]$
  \begin{equation} \label{sup4}
\|x\|_{L^{\infty}(0,T; H)}   \leq   \|x_0\| +   \sqrt{ \frac{T}{c_0}}\left(  ( \varphi + \psi)(x_0) - \inf_H ( \varphi + \psi)    \right)^{\frac{1}{2}} .
   \end{equation}
\end{proof}

Let us now exploit another a priori energy estimate.

\begin{proposition}  Let $\left(x,\upsilon\right)$ be  the strong solution of system (\ref{basica})-(\ref{basicc}) on $\left[0,T\right]$, $T>0$. Then 
  \label{pr:energy20}
\begin{align}
\int_0^T  \|\dot v(t)\|^2 dt   & \leq  \| v_0\|^2 +  2T \| \nabla \psi(x_0) \|^2 + \frac{2T^2 L_{\psi}^2}{c_0}\left(  ( \varphi + \psi)(x_0) - \inf_H ( \varphi + \psi) \right)\\
\|v\|_{L^{\infty}(0,T; H)}     & \leq  \| v_0\| +  \sqrt{2T} \| \nabla \psi(x_0) \| +    \sqrt{\frac{2}{c_0}} TL_{\psi}\left(  ( \varphi + \psi)(x_0) - \inf_H ( \varphi + \psi) \right)^{\frac{1}{2}}.
\end{align}
\end{proposition}
\begin{proof}
By taking the scalar product of (\ref{basicb}) by $\dot{v}(t)$ we obtain 
 \[
 \lambda(t)\inner{\dot x (t)}{\dot v(t) }   +  \|\dot v(t)\|^2   +   \inner{ v(t) } {\dot v (t)}  + \inner{ \nabla \psi(x(t))}{\dot v (t)}=0.
  \]
 By  (\ref{energy2}) we infer 
 \begin{equation} \label{energy30}
 \|\dot v(t)\|^2 +     \inner{ v(t) } {\dot v (t)}  + \inner{ \nabla \psi(x(t))}{\dot v (t)} \leq 0.
  \end{equation}
Hence
 \begin{align} \label{energy31}
 \|\dot v(t)\|^2 +     \frac{1}{2}  \frac{d}{dt} \| v(t)\|^2  &\leq    \| \nabla \psi(x(t)) \|   \| \dot v (t) \| \\ 
 & \leq \frac{1}{2}  \| \nabla \psi(x(t)) \|^2    + \frac{1}{2}   \| \dot v (t) \|^2
  \end{align}
which implies
\begin{equation} \label{energy320}
 \|\dot v(t)\|^2 +       \frac{d}{dt} \| v(t)\|^2   \leq   \| \nabla \psi(x(t)) \|^2.
  \end{equation}
 By integrating  the above inequality we deduce that, for any $t\in [0,T]$
\begin{equation} \label{energy32}
 \int_0^t \|\dot v(\tau)\|^2 d \tau +    \| v(t)\|^2   \leq   \| v_0\|^2   +  \int_0^T  \| \nabla \psi(x(\tau)) \|^2 d \tau.
  \end{equation}
Since $\nabla \psi$ is $L_{\psi}$-Lipschitz continuous
\begin{equation} \label{energy33}
 \| \nabla \psi(x(\tau)) \| \leq \| \nabla \psi(x_0) \| + L_{\psi} \| x(\tau) - x_0\|.
  \end{equation}
A careful look at the proof of (\ref{sup4}) gives the more precise estimate
  \begin{equation} \label{sup44}
\|x -x_0\|_{L^{\infty}(0,T; H)}   \leq   \sqrt{ \frac{T}{c_0}}\left(  ( \varphi + \psi)(x_0) - \inf_H ( \varphi + \psi)    \right)^{\frac{1}{2}} 
   \end{equation}
Hence for any $\tau \in [0,T]$
  \begin{equation} \label{sup45}
\| \nabla \psi(x(\tau)) \| \leq \| \nabla \psi(x_0) \| + L_{\psi} \sqrt{ \frac{T}{c_0}}\left(  ( \varphi + \psi)(x_0) - \inf_H ( \varphi + \psi)    \right)^{\frac{1}{2}}.
   \end{equation}  
Combining (\ref{energy32}) with (\ref{sup45}) gives
\begin{equation} \label{energy34}
 \int_0^t \|\dot v(\tau)\|^2 d \tau + \| v(t)\|^2   \leq   \| v_0\|^2 +  2T \| \nabla \psi(x_0) \|^2 + \frac{2T^2 L_{\psi}^2}{c_0}\left(  ( \varphi + \psi)(x_0) - \inf_H ( \varphi + \psi) \right) .
  \end{equation}
As a consequence
\begin{equation} \label{energy35}
\int_0^T  \|\dot v(t)\|^2 dt    \leq  \| v_0\|^2 +  2T \| \nabla \psi(x_0) \|^2 + \frac{2T^2 L_{\psi}^2}{c_0}\left(  ( \varphi + \psi)(x_0) - \inf_H ( \varphi + \psi) \right),
  \end{equation}
and 
\begin{equation} \label{energy36}
\|v\|_{L^{\infty}(0,T; H)}      \leq  \| v_0\| +  \sqrt{2T} \| \nabla \psi(x_0) \| +    \sqrt{\frac{2}{c_0}} TL_{\psi}\left(  ( \varphi + \psi)(x_0) - \inf_H ( \varphi + \psi) \right)^{\frac{1}{2}}.
  \end{equation}
which ends the proof.
\end{proof}

We can now deduce from the two preceding propositions an a priori bound on the $L^{\infty}$ norm of $\dot x$ and $\dot v$. 

\begin{proposition} 
  \label{pr:energy30}
\begin{align}
\|\dot x\|_{L^{\infty}(0,T; H)} &\leq  \frac{\| v_0\|}{c_0} +   \frac{1+ \sqrt{2T}}{c_0}   \| \nabla \psi(x_0) \| +   \frac{(\sqrt{2}T+ \sqrt{T}) L_{\psi}}{c_0^{\frac{3}{2}}} \left(  ( \varphi + \psi)(x_0) - \inf_H ( \varphi + \psi)    \right)^{\frac{1}{2}}\\  
\|\dot v\|_{L^{\infty}(0,T; H)} &\leq  \| v_0\| +   (1+ \sqrt{2T})   \| \nabla \psi(x_0) \| +   \frac{(\sqrt{2}T+ \sqrt{T}) L_{\psi}}{c_0^{\frac{1}{2}}} \left(  ( \varphi + \psi)(x_0) - \inf_H ( \varphi + \psi)    \right)^{\frac{1}{2}}
\end{align}
\end{proposition}
\begin{proof}
a) Let us return to the equation obtained by taking
  the inner product of both sides of (\ref{basicb}) by $\dot{x}(t)$ 
  \[
 \lambda(t)\|\dot x(t)\|^2 + \inner{\dot x (t)}{\dot v(t) }   + \inner{ v(t) } {\dot x (t)}  + \inner{ \nabla \psi(x(t))}{\dot x (t)}=0.
  \] 
  By  (\ref{energy2}), and $\lambda$ is minorized  by the  positive constant $c_0$ on $[0,T], $we infer 
\begin{equation*} 
c_0\|\dot x(t)\|^2   + \inner{ v(t) + \nabla \psi(x(t))} {\dot x (t)} \leq 0.
  \end{equation*}
 Hence
 \begin{align*} 
\|\dot x(t)\|    &\leq \frac{1}{c_0} \| v(t) + \nabla \psi(x(t))\|\\
& \leq \frac{1}{c_0} (\| v(t)\|  + \|\nabla \psi(x_0)\| + L_{\psi} \|   x(t) -x_0\|).  
  \end{align*} 
 By  (\ref{sup44})   and   (\ref{energy36}) we deduce that
 \begin{align*} 
\|\dot x\|_{L^{\infty}(0,T; H)} &\leq  \frac{1}{c_0} \left(  \| v_0\| +  \sqrt{2T} \| \nabla \psi(x_0) \| +    \sqrt{\frac{2}{c_0}} TL_{\psi}\left(  ( \varphi + \psi)(x_0) - \inf_H ( \varphi + \psi) \right)^{\frac{1}{2}}  \right) \\    
&+  \frac{1}{c_0}  \|\nabla \psi(x_0)\| +  
 \frac{L_{\psi}}{c_0}  \sqrt{ \frac{T}{c_0}}\left(  ( \varphi + \psi)(x_0) - \inf_H ( \varphi + \psi)    \right)^{\frac{1}{2}}     \\
& \leq \frac{\| v_0\|}{c_0} +   \frac{1+ \sqrt{2T}}{c_0}   \| \nabla \psi(x_0) \| +   \frac{(\sqrt{2}T+ \sqrt{T}) L_{\psi}}{c_0^{\frac{3}{2}}} \left(  ( \varphi + \psi)(x_0) - \inf_H ( \varphi + \psi)    \right)^{\frac{1}{2}}.             
  \end{align*} 
b)  Let us return to the equation obtained by taking
  the inner product of both sides of (\ref{basicb}) by $\dot{v}(t)$  
  \[
 \lambda(t)\inner{\dot x (t)}{\dot v(t) }   +  \|\dot v(t)\|^2   +   \inner{ v(t) } {\dot v (t)}  + \inner{ \nabla \psi(x(t))}{\dot v (t)}=0.
  \]
  A similar argument as above yields
   \begin{equation*} 
\|\dot v(t)\|    \leq  \| v(t) + \nabla \psi(x(t))\|\\
  \end{equation*} 
 from which we deduce the result. 
\end{proof}

Let us enunciate some straight  consequences of Proposition \ref{pr:energy30} .
 
\begin{corollary}
  \label{cr:v.1}
 The following properties hold: for any $0 < T < + \infty$
  \begin{enumerate}
  \item $t\mapsto x(t)$ is Lipschitz continuous on $[0,T]$ with constant 
  $$\frac{\| v_0\|}{c_0} +   \frac{1+ \sqrt{2T}}{c_0}   \| \nabla \psi(x_0) \| +   \frac{(\sqrt{2}T+ \sqrt{T}) L_{\psi}}{c_0^{\frac{3}{2}}} \left(  ( \varphi + \psi)(x_0) - \inf_H ( \varphi + \psi)    \right)^{\frac{1}{2}}$$
  \item   $t\mapsto v(t)$ is Lipschitz continuous on $[0,T]$, with constant 
    $$\| v_0\| +   (1+ \sqrt{2T})   \| \nabla \psi(x_0) \| +   \frac{(\sqrt{2}T+ \sqrt{T}) L_{\psi}}{c_0^{\frac{1}{2}}} \left(  ( \varphi + \psi)(x_0) - \inf_H ( \varphi + \psi)    \right)^{\frac{1}{2}}$$
  \end{enumerate}
\end{corollary}

\section{{\large Stability results}}

In the next theorem we analyze the Lipschitz continuous dependence of  the solution $\left(x,\upsilon\right)$
of the Cauchy problem (\ref{basica})-(\ref{basicc})
with respect to the function $\lambda$ and the initial point $\left(x_{0},\upsilon_{0}\right)$.
Our demonstration is based on that followed in \cite [Theorem 3.1]{ARS}.

\begin{theorem}\label{basic-stab}
Suppose that $\lambda,\eta:\left[0,T\right]\longrightarrow\left[c_{0},+\infty\right[$
are absolutely continuous functions, with $T>0$ and $c_{0}>0$. Let
$\left(x,\upsilon\right),\left(y,w\right):\left[0,T\right]\longrightarrow H\times H$
be the respective strong solutions of the inclusions 
\begin{align}
 & \lambda\dot{x}+\dot{\upsilon}+\upsilon+  \nabla \psi \left(x\right)=0,\:\,\upsilon\in \partial \varphi \left(x\right),\:\, x\left(0\right)=x_{0},\:\,\upsilon\left(0\right)=\upsilon_{0},\label{6}\\
 & \eta\dot{y}+\dot{w}+w+ \nabla \psi\left(y\right) =0,\:\, w\in \partial \varphi \left(y\right),\:\, y\left(0\right)=y_{0},\:\, w\left(0\right)=w_{0}.\label{7}
\end{align}
Define $\theta:\left[0,T\right]\longrightarrow\mathbb{R}$, for each $t\in  \left[0,T\right]$, by
\[
\theta (t)=\sqrt{c_{0}^{2}\left\Vert x(t)-y(t)\right\Vert ^{2}+\left\Vert \upsilon (t)-w (t) \right\Vert ^{2}}.
\]
Then 
\begin{equation}\label{basic-stab-0}
\| \theta \|_{L^{\infty}([0,T])} \leq  \left[ \frac{\lambda (0)+\eta (0)}{2}\left\Vert x_{0}-y_{0}\right\Vert +\left\Vert \upsilon_{0}-w_{0}\right\Vert  + \frac{C}{2} \| \lambda-\eta \|_{L^1([0,T])} \right]
\times  \exp \left(  \frac{ \| \dot{\lambda} + \dot{\eta} \|_{L^1}   }{2c_0} + T(1 + \frac{L_{\psi}}{c_0})  \right),
\end{equation}
with 
\begin{align*}
C &=  \frac{\| v_0\| + \| w_0\|}{c_0} +   \frac{1+ \sqrt{2T}}{c_0}  ( \| \nabla \psi(x_0) \| + \| \nabla \psi(y_0) \|) + \\
 & \frac{(\sqrt{2}T+ \sqrt{T}) L_{\psi}}{c_0^{\frac{3}{2}}}
\left( \left(  ( \varphi + \psi)(x_0) - \inf_H ( \varphi + \psi)    \right)^{\frac{1}{2}} +   \left(  ( \varphi + \psi)(y_0) - \inf_H ( \varphi + \psi)    \right)^{\frac{1}{2}} \right).
\end{align*}
In particular, if $x_{0}=y_{0}$, $\upsilon_{0}=w_{0}$, then 
\begin{equation}\label{basic-stab-est}
\| \theta \|_{L^{\infty}([0,T])} \leq   \frac{C}{2} \| \lambda-\eta \|_{L^1 ([0,T])} 
\times  \exp \left(  \frac{ \| \dot{\lambda} + \dot{\eta} \|_{L^1 ([0,T])}   }{2c_0} + T(1 + \frac{L_{\psi}}{c_0})  \right),
\end{equation}
with
\begin{align*}
C =  \frac{2 \| v_0\| }{c_0} +   2\frac{1+ \sqrt{2T}}{c_0} \| \nabla \psi(x_0) \|  + 
 2\frac{(\sqrt{2}T+ \sqrt{T}) L_{\psi}}{c_0^{\frac{3}{2}}} \left(( \varphi + \psi)(x_0) - \inf_H ( \varphi + \psi)    \right)^{\frac{1}{2}}. 
\end{align*}
\end{theorem}
\begin{proof}
To simplify the exposition, define $\gamma:\left[0,T\right]\longrightarrow\mathbb{R}$
\[
\gamma=\frac{\lambda+\eta}{2}.
\]
Using the assumptions $\lambda, \eta \geq c_0$ and the monotonicity of
$\partial \varphi$  we conclude that for any $t\in\left[0,T\right]$
\[
c_{0}\leq\gamma,\,\left\langle x-y,\upsilon-w\right\rangle \geq0.
\]
Therefore
\begin{equation}
\theta\leq\sqrt{\gamma^{2}\left\Vert x-y\right\Vert ^{2}+\left\Vert \upsilon-w\right\Vert ^{2}}\leq\left\Vert \gamma\left(x-y\right)+\left(\upsilon-w\right)\right\Vert .\label{eq:10}
\end{equation}
Let us show that $\theta$ satisfies a Gronwall inequality. Let us start from
\begin{equation*}
\frac{d}{dt}\left[\gamma\left(x-y\right)+\left(\upsilon-w\right)\right]=\dot{\gamma}\left(x-y\right)+ \gamma\left(\dot{x} - \dot{y}\right)+\left(\dot{\upsilon}-\dot{w}\right).
\end{equation*}
In view of (\ref{6}) and (\ref{7}), 
\[
\left(\dot{\upsilon}-\dot{w}\right)=-\lambda\dot{x}+\eta\dot{y}-\left(\upsilon-w\right)-\left(\nabla \psi\left(x\right)-\nabla \psi\left(y\right)\right).
\]
Combining the two above relations gives
\begin{align*}
\frac{d}{dt}\left[\gamma\left(x-y\right)+\left(\upsilon-w\right)\right]&=   \dot{\gamma}\left(x-y\right) + \gamma\left(\dot{x} - \dot{y}\right) -\lambda\dot{x}+\eta\dot{y}-\left(\upsilon-w\right)-\left(\nabla \psi\left(x\right)-\nabla \psi\left(y\right)\right),\\
   &=  \dot{\gamma}\left(x-y\right) + \frac{\eta - \lambda}{2}\left(\dot{x} + \dot{y}     \right) - \left(\upsilon -w \right)- \left(\nabla \psi\left(x\right)-\nabla \psi\left(y\right)\right).
\end{align*}
Since $\gamma,x,y,w,\upsilon$
are absolutely continuous, the function $\gamma\left(x-y\right)+\left(\upsilon-w\right)$
is also absolutely continuous. As a consequence, integration of the
above inequality on $\left[0,s\right]$, for $s\in\left[0,T\right]$,
yields
\begin{align*}
&\left[\gamma\left(x-y\right)+\left(\upsilon-w\right)\right]\left(s\right)-\left[\gamma\left(x-y\right)+ \left(\upsilon-w\right)\right]\left(0\right)  =\\
 & \frac{1}{2}\int_{0}^{s}\left(\eta - \lambda\right)\left(\dot{x}+ \dot{y}\right)dt + \int_{0}^{s}\left(\dot{\gamma}\left(x-y\right)-\left(\upsilon-w\right)\right)dt
-\int_{0}^{s}\left(\nabla \psi\left(x\right)-\nabla \psi\left(y\right)\right)dt.
\end{align*}
Passing to the norm
\begin{align}\label{stab9}
&\left\Vert \left[\gamma\left(x-y\right)+\left(\upsilon-w\right)\right](s)\right\Vert \nonumber 
\leq  \gamma\left(0\right)\left\Vert x_{0}-y_{0}\right\Vert +\left\Vert \upsilon_{0}-w_{0}\right\Vert \nonumber \\
 &  \quad \quad + \frac{1}{2}\int_{0}^{s}|\lambda-\eta|\left\Vert \dot{x} + \dot{y}\right\Vert dt+\int_{0}^{s} \|\dot{\gamma}\left(x-y\right)-\left(\upsilon-w\right)\|dt
 +\int_{0}^{s}\|\nabla \psi\left(x\right)-\nabla \psi\left(y\right)\|dt.
\end{align}
By  Lipschitz continuity property of $\nabla \psi$, and definition of $\theta$, we have 
\begin{align}\label{stab10}
\|\nabla \psi\left(x(t)\right)-\nabla \psi\left(y(t)\right)\| &\leq L_{\psi} \|x(t)-y(t) \| \nonumber  \\
& \leq \frac{L_{\psi}}{c_0} \theta (t).
\end{align}
On the other hand 
\begin{align}
\left\Vert \dot{\gamma}\left(x-y\right)-\left(\upsilon-w\right)\right\Vert ^{2} & =|\dot{\gamma}| ^{2}\left\Vert x-y\right\Vert ^{2}+\left\Vert \upsilon-w\right\Vert ^{2}+2\left\langle c_{0}\left(x-y\right),\frac{\dot{\gamma}}{c_{0}}\left(w-\upsilon\right)\right\rangle \nonumber \\
 & \leq | \dot{\gamma}| ^{2}\left\Vert x-y\right\Vert ^{2}+\left\Vert \upsilon-w\right\Vert ^{2}+c_{0}^{2}\left\Vert x-y\right\Vert ^{2}+\frac{\left|\dot{\gamma}\right|^{2}}{c_{0}^{2}}\left\Vert \upsilon-w\right\Vert ^{2}\nonumber \\
 & =\left(\frac{\left|\dot{\gamma}\right|^{2}}{c_{0}^{2}}+1\right)\theta^{2}, \label{stab11}
\end{align}
Combining (\ref{stab9}),  (\ref{stab10}),  (\ref{stab11})  and \   $\theta \leq\left\Vert \gamma\left(x-y\right)+\left(\upsilon-w\right)\right\Vert$, we obtain
\begin{align*}
\theta(s) &\leq  \gamma\left(0\right)\left\Vert x_{0}-y_{0}\right\Vert +\left\Vert \upsilon_{0}-w_{0}\right\Vert  +  \frac{1}{2}\int_{0}^{s}|\lambda-\eta|\left\Vert \dot{x} + \dot{y}\right\Vert dt+ \int_{0}^{s} \theta (t) \left[\sqrt{  \frac{\left|\dot{\gamma}\right|^{2}}{c_{0}^{2}}+1} + \frac{L_{\psi}}{c_0} \right] dt  \nonumber  \\
 & \leq  \gamma\left(0\right)\left\Vert x_{0}-y_{0}\right\Vert +\left\Vert \upsilon_{0}-w_{0}\right\Vert  + \frac{1}{2} \| \lambda-\eta \|_{L^1 ([0,T])}  \| \dot{x} + \dot{y}\|_{L^{\infty}([0,T])}
 + \int_{0}^{s} \theta (t) \left[  \frac{\left|\dot{\gamma}\right|}{c_0}+1 + \frac{L_{\psi}}{c_0}  \right]dt.  
\end{align*}
Applying Gronwall's inequality yields
\begin{equation*}
\theta(s) \leq \left[ \gamma\left(0\right)\left\Vert x_{0}-y_{0}\right\Vert +\left\Vert \upsilon_{0}-w_{0}\right\Vert  + \frac{1}{2} \| \lambda-\eta \|_{L^1([0,T])}  \| \dot{x} + \dot{y}\|_{L^{\infty}([0,T])}\right] \times \exp \int_0^s \left[  \frac{\left|\dot{\gamma}\right|}{c_0}+1 + \frac{L_{\psi}}{c_0}  \right]dt.
\end{equation*}
Combining this estimation with the bound on the $L^{\infty}$ norm of $\dot{x}$ and  $\dot{y}$ (see Proposition (\ref{pr:energy30})) gives
\begin{equation*}
\| \theta \|_{L^{\infty}([0,T])} \leq  \left[ \gamma\left(0\right)\left\Vert x_{0}-y_{0}\right\Vert +\left\Vert \upsilon_{0}-w_{0}\right\Vert  + \frac{C}{2} \| \lambda-\eta \|_{L^1([0,T])} \right]
\times  \exp \int_0^T \left[  \frac{\left|\dot{\gamma}\right|}{c_0}+1 + \frac{L_{\psi}}{c_0}  \right]dt
\end{equation*}
with
\begin{align*}
C &=  \frac{\| v_0\| + \| w_0\|}{c_0} +   \frac{1+ \sqrt{2T}}{c_0}  ( \| \nabla \psi(x_0) \| + \| \nabla \psi(y_0) \|) + \\
 & \frac{(\sqrt{2}T+ \sqrt{T}) L_{\psi}}{c_0^{\frac{3}{2}}}
\left( \left(  ( \varphi + \psi)(x_0) - \inf_H ( \varphi + \psi)    \right)^{\frac{1}{2}} +   \left(  ( \varphi + \psi)(y_0) - \inf_H ( \varphi + \psi)    \right)^{\frac{1}{2}} \right).
\end{align*}
Equivalently
\begin{equation*}
\| \theta \|_{L^{\infty}([0,T])} \leq  \left[\frac{\lambda (0)+\eta (0)}{2}\left\Vert x_{0}-y_{0}\right\Vert +\left\Vert \upsilon_{0}-w_{0}\right\Vert  + \frac{C}{2} \| \lambda-\eta \|_{L^1([0,T])} \right]
\times  \exp \left(  \frac{ \| \dot{\lambda} + \dot{\eta} \|_{L^1([0,T])}   }{2c_0} + T(1 + \frac{L_{\psi}}{c_0})  \right)
\end{equation*}
\end{proof}

\begin{remark} { \rm
In the cas $\psi= 0$ we   recover exactly the same stability results as \cite [Theorem 3.1]{ARS}.}
\end{remark}

\begin{remark} { \rm
It is worth noticing that, besides the Lipschitz continuous dependence
with respect to $\lambda$ of the solution $(x,\upsilon)$ of (\ref{basica})\textendash{}(\ref{basicc}),
Theorem (\ref{basic-stab}) also provides its continuous dependence with respect to
the initial data ($x_{0},\upsilon_{0}$).
More precisely, if $\left(x_{n},\upsilon_{n}\right)$ (resp. $\left(x,\upsilon\right)$
) is the solution of (\ref{basica})\textendash{}(\ref{basicc}) corresponding to the Cauchy data $\left(x_{0n},\upsilon_{0n}\right)$(resp.
$\left(x_{0},\upsilon_{0}\right)$), as a direct consequence of (\ref{basic-stab-0})
we obtain for all $T>0$
\[
\left(\partial\varphi\ni\left(x_{0n},\upsilon_{0n}\right)\longrightarrow\left(x_{0},\upsilon_{0}\right)\right)\Longrightarrow\left(x_{n},\upsilon_{n}\right)\longrightarrow\left(x,\upsilon\right)\;\textrm{uniformly on}\:\left[0,T\right].
\]

Note also that, by contrast with semigroup generated by a subdifferential
convex \cite {Br}, there is no regularizing effect on the initial condition:
there is no way to define a solution of (\ref{basica})\textendash{}(\ref{basicc}) for $x_{0}\in\overline{\textrm{dom}\partial\varphi}\setminus\textrm{dom}\partial\varphi$
since in that case, for any approximation sequence $x_{0n}\in\textrm{dom}\partial\varphi$
and $\upsilon_{0n}\in\partial\varphi\left(x_{0n}\right)$, one has
$\underset{n}{\lim}\left\Vert \upsilon_{0n}\right\Vert =+\infty$,
which would imply the blow-up of the sequence $\left(x_{n},\upsilon_{n}\right)$
as $n\longrightarrow+\infty,$ on any finite interval.}
\end{remark}

\begin{remark} { \rm
Another  approach is to study the  equivalent  problem (\ref{basic.z1})-(\ref{basic.z2}), based  on the known stability results for the Cauchy-Lipschitz problem. Although conceptually simple,  this approach seems more technical.}
\end{remark}

\section{Bounded Variation Regularization Coefficient $\lambda\left(\cdot\right)$ }

Let us suppose that $\lambda\left(\cdot\right):\left[0,T\right]\longrightarrow\left]0,\infty\right[$
is of bounded variation on $\left[0,T\right]$, where $T>0$. That
is $\textrm{TV}\left(\lambda,\left[0,T\right]\right)<+\infty$, where
$\textrm{TV}\left(\lambda,\left[0,T\right]\right)$ is the total variation
of $\lambda$ on $\left[0,T\right]:$
\[
\textrm{TV}\left(\lambda,\left[0,T\right]\right)=\sup\sum_{i=1}^{p}\left|\lambda\left(\tau_{i}\right)-\lambda\left(\tau_{i-1}\right)\right|,
\]
the supremum  being taken over all $p\in\mathbb{N}$ and all strictly
increasing sequences $\tau_{0}<\tau_{1}<\cdot\cdot\cdot<\tau_{p}$ of points of $\left[0,T\right]$.
Function $\lambda$ may involve jumps. We also suppose that $\lambda$
is bounded away from $0$:
\[
\inf\lambda\left(\left[0,T\right]\right)>0.
\]
The following lemmas which are proved in \cite[Lemma 3.1, Lemma 3.2]{ARS}
 gather some classical facts concerning the approximation of functions
of bounded variation by smooth functions together with some technical
results useful for sequel. 
\begin{lemma} \label{BV-lemma1}
Let $\lambda:\left[0,T\right]\longrightarrow\left]0,\infty\right[$
be of bounded variation on $\left[0,T\right]$. Then there exists
a sequence $\left(\lambda_{n}\right)_{n\in\mathbb{N}}$, with $\lambda_{n}\in C^{\infty}\left(\left[0,T\right]\right)$
for each $n\in\mathbb{N}$, such that 

\medskip

$i)$ $\inf\lambda(\left[0,T\right])\leq\lambda_{n}\left(t\right)\leq\sup\lambda(\left[0,T\right])$,
$\forall t\in\left[0,T\right]$, $\forall n\in\mathbb{N}$.

In particular, $\lambda_{n}\geq0$ if $\lambda\geq0$;

\medskip

$ii)$ $\lambda_{n}\longrightarrow\lambda$ in $L^{p}\left(0,T\right)$
for any $1\leq p<\infty$;

\medskip

$iii)$ $\textrm{TV}\left(\lambda_{n},\left[0,T\right]\right)=\int_{0}^{T} |\dot{\lambda}_{n}\left(t\right)| dt\leq\textrm{TV}\left(\lambda,\left[0,T\right]\right).$
\end{lemma}
$\vphantom{}$
\begin{lemma}\label{BV-lemma2}
Let $z_{n},z\in C\left(\left[0,T\right],H\right)$ be such that $z_{n}\longrightarrow z$
uniformly and $\left(z_{n}\right)_{n}$ is L-Lipschitz continuous
for some positive constant $L$ independent of $n\in\mathbb{N}$.
Let $\lambda_{n}\longrightarrow\lambda$ in $L^{2}\left(0,T\right).$
Then $\lambda_{n}\dot{z}_{n}$ converges weakly to $\lambda\dot{z}$
in $L^{2}\left(0,T;H\right).$ 
\end{lemma}
$\vphantom{}$

We can now state the main result of this section.

\begin{theorem} \label{BV-Theorem}
Let $\lambda:\left[0,T\right]\longrightarrow\left]0,\infty\right[$
be of bounded variation on $\left[0,T\right]$, and suppose $c_{0}=\inf\lambda\left(\left[0,T\right]\right)>0$.
Let  $ x_{0} \in \mbox{dom} \partial \varphi$ and $\upsilon_{0} \in \partial  \varphi (x_{0})$, \ $\upsilon_{0}\neq 0$. Then there is existence and uniqueness of a strong solution $\left(x,\upsilon\right):\left[0,T\right]\longrightarrow H\times H$
 of the Cauchy problem 
\begin{subequations}
\begin{align}
 & \upsilon\left(t\right)\in \partial \varphi \left(x\left(t\right)\right),\qquad0\leq t\leq T\label{15}\\
 & \lambda\left(t\right)\dot{x}\left(t\right)+\dot{\upsilon}\left(t\right)+\upsilon\left(t\right)+ \nabla  \psi\left(x\left(t\right)\right)=0,\quad a.e. \  \:0\leq t\leq T,\label{16}\\
 & x\left(0\right)=x_{0},\quad\upsilon\left(0\right)=\upsilon_{0}.\label{17}
\end{align}
\end{subequations}\end{theorem}

\begin{proof} \
\emph{Existence: }According to Lemma \ref{BV-lemma1}, there exists a sequence of functions 
$\left(\lambda_{n}\right)_{n\in\mathbb{N}}$ \  in $C^{\infty}\left(\left[0,T\right]\right)$, with $\lambda_{n}(t) \geq c_{0}$,
which converges to $\lambda$ in $L^{p}\left(0,T\right)$
for $p\geq1$, and 
satisfies the condition 
\begin{equation}
\int_{0}^{T}\left|\dot{\lambda}_{n}\left(t\right)\right|dt\leq\textrm{TV}\left(\lambda\right).\label{18}
\end{equation}
For each $n \in \mathbb N$ there exists a unique $\left(x_{n},\upsilon_{n}\right)$
solution of the differential inclusion
\begin{subequations}
\begin{align}
 & \upsilon_{n}\left(t\right)\in \partial \varphi\left(x_{n}\left(t\right)\right),\quad0\leq t\leq T,\label{19}\\
 & \lambda_{n}\left(t\right)\dot{x}_{n}\left(t\right)+\dot{\upsilon_{n}}\left(t\right)+\upsilon_{n}\left(t\right)+\nabla  \psi \left(x_{n}\left(t\right)\right)=0,\quad a.e.  \ \:0\leq t\leq T,\label{20}\\
 & x_{n}\left(0\right)=x_{0},\quad\upsilon_{n}\left(0\right)=\upsilon_{0}.\label{21}
\end{align}
\end{subequations}
We will show that $\left(x_{n},\upsilon_{n}\right)$ converges uniformly
to a solution of (\ref{15}-\ref{17}).

\noindent Consider, in $C\left(\left[0,T\right],H\times H\right)$, the norm
\[
\left\Vert \left(z,w\right)\right\Vert _{c_{0}}=\underset{t\in\left[0,T\right]}{\max}\sqrt{c_{0}^{2}\left\Vert z\left(t\right)\right\Vert ^{2}+\left\Vert w\left(t\right)\right\Vert ^{2}.}
\]
It is equivalent to the $sup$ norm in $C\left(\left[0,T\right],H\times H\right).$
Using Theorem \ref{basic-stab},  (\ref{basic-stab-est}) and (\ref{18}) we deduce the existence of a constant $C$ (which is independant of $n$) such that
for any $n,m$  we have 
\begin{align*}
\left\Vert \left(x_{n},\upsilon_{n}\right)-\left(x_{m},\upsilon_{m}\right)\right\Vert_{c_{0}}  & \leq C \exp\left(\frac{\left\Vert \dot{\lambda}_{n}+\dot{\lambda}_{m}\right\Vert _{L^{1}\left(\left[0,T\right]\right)}}{2c_{0}}+ T(1 + \frac{L_{\psi}}{c_0}) \right)\left\Vert \lambda_{n}-\lambda_{m}\right\Vert _{L^{1}\left(\left[0,T\right]\right)}\\
 & \leq C\exp\left(\frac{TV\left(\lambda\right)}{c_{0}}+T(1 + \frac{L_{\psi}}{c_0})  \right)\left\Vert \lambda_{n}-\lambda_{m}\right\Vert _{L^{1}\left(\left[0,T\right]\right)}.
\end{align*}
Since  the sequence 
$\left(\lambda_{n}\right)_{n\in\mathbb{N}}$
 converges to $\lambda$ in $L^{1}\left(0,T\right)$, we deduce that
$\left(x_{n},\upsilon_{n}\right)$ is a Cauchy sequence with
respect to the $sup$ norm. Therefore, $\left(x_{n},\upsilon_{n}\right)$
converges uniformly to some continuous $\left(x,\upsilon\right):\left[0,T\right]\longrightarrow H\times H$.
Since $\nabla  \psi$ is continuous we also obtain 
$$
\nabla  \psi (x_{n}) \rightarrow \nabla  \psi (x)   \quad \mbox{uniformly on  }  \  [0,T].
$$
Moreover, Lemma \ref{BV-lemma2} ensures that $\lambda_{n}\dot{x}_{n}$ and $\dot{v}_{n}$ converge weakly to $\lambda\dot{x}$ and $\dot{v}$ in $L^{2}\left(0,T;\ H\right)$.
Hence letting $n\rightarrow +\infty$ in (\ref{20}) gives 
$$
\lambda\dot{x}+\dot{\upsilon}+\upsilon+ \nabla  \psi\left(x\right)=0,
$$
that's (\ref{16}). Finally, since the graph of $\partial \varphi$ is closed we obtain, for all $t \in [0,T]$, 
$\upsilon\left(t\right)\in \partial \varphi \left(x\left(t\right)\right)$
that's (\ref{15}).
Hence $\left(x,\upsilon\right)$  is a solution of (\ref{15})-(\ref{17}).

\medskip

\emph{Uniqueness: } We adapt the proof of \cite{ARS}, using differential
and integral calculus for BV functions which involves differential
measures. 

Define $\lambda^{-}:\left[0,T\right]\longmapsto\left[c_{0},+\infty\right[$
by 
\begin{align*}
\lambda^{-}\left(0\right) & =\lambda\left(0\right),\\
 & 0<t\leq T:\lambda^{-}\left(t\right)=\underset{\varepsilon>0,\varepsilon\rightarrow0}{\lim\lambda\left(t-\varepsilon\right).}
\end{align*}
Let $\left(x,\upsilon\right),\left(y,w\right):\left[0,T\right]\rightarrow H\times H$
be two strong  solutions of (\ref{15})-(\ref{16})-(\ref{17}). Explicitly 
\begin{align*}
 & \lambda\dot{x}+\dot{\upsilon}+\upsilon+\nabla\psi\left(x\right)=0\: \ a.e.;\;\upsilon\left(t\right)\in\partial\varphi\left(x\left(t\right)\right)\;\forall t;\; x\left(0\right)=x_{0},\;\upsilon\left(0\right)=\upsilon_{0},\\
 & \lambda\dot{y}+\dot{w}+w+\nabla\psi\left(y\right)=0\: \ a.e.;\; w\left(t\right)\in\partial\varphi\left(y\left(t\right)\right)\;\forall t;\; y\left(0\right)=x_{0},\; w\left(0\right)=\upsilon_{0}.
\end{align*}
Since $\lambda=\lambda^{-}$ a.e., we also have 
\[
\lambda^{-}\dot{x}+\dot{\upsilon}+\upsilon+\nabla\psi\left(x\right)=0\;\textrm{and}\;\lambda^{-}\dot{y}+\dot{w}+w+\nabla\psi\left(y\right)=0\; \: a.e.
\]
and consequently 
\begin{equation}
\lambda^{-}\left(\dot{x}-\dot{y}\right)+\left(\dot{\upsilon}-\dot{w}\right)+\upsilon-w+\nabla\psi\left(x\right)-\nabla\psi\left(y\right)=0\; \: a.e.\label{48}
\end{equation}
In terms of differential measures on $\left[0,T\right]$ we have  \cite[Proposition 11.1]{Moreau}
\begin{equation}
d\left[\lambda^{-}\left(x-y\right)+\left(\upsilon-w\right)\right]=\lambda^{-}d\left(x-y\right)+\left(x-y\right)d\lambda^{-}+d\left(\upsilon-w\right).\label{49}
\end{equation}
Integrating the left hand term on $\left[0,s\right[$ and taking the
initial condition into account, we obtain for $s\in\left[0,T\right]$  \cite[Corollary 8.2]{Moreau}
\begin{equation}
\int_{\left[0,s\right[}d\left[\lambda^{-}\left(x-y\right)+\left(\upsilon-w\right)\right]=\lambda^{-}\left(s\right)\left(x\left(s\right)-y\left(s\right)\right)+\left(\upsilon\left(s\right)-w\left(s\right)\right).\label{50}
\end{equation}
Now integrating the right hand term of (\ref{49}) on $\left[0,s\right)$
and taking (\ref{48}) into account, we get 
\begin{align}
\int_{\left[0,s\right[}\left[\lambda^{-}d\left(x-y\right)+\left(x-y\right)d\lambda^{-}+d\left(\upsilon-w\right)\right]\nonumber \\
 & =\int_{\left[0,s\right[}\left[\lambda^{-}d\left(x-y\right)+d\left(\upsilon-w\right)\right]+\int_{\left[0,s\right[}\left(x-y\right)d\lambda^{-}\nonumber \\
 & =\int_{\left[0,s\right[}\left[\lambda^{-}\left(\dot{x}-\dot{y}\right)+\dot{\upsilon}-\dot{w}\right]dt+\int_{\left[0,s\right[}\left(x-y\right)d\lambda^{-}\nonumber \\
 & =-\int_{\left[0,s\right[}\left[\upsilon-w+\nabla\psi\left(x\right)-\nabla\psi\left(y\right)\right]dt+\int_{\left[0,s\right[}\left(x-y\right)d\lambda^{-}.\label{51}
\end{align}
From (\ref{49},\ref{50},\ref{51}) we deduce 
\[
\lambda^{-}\left(s\right)\left(x\left(s\right)-y\left(s\right)\right)+\left(\upsilon\left(s\right)-w\left(s\right)\right)=\int_{\left[0,s\right[}\left(x-y\right)d\lambda^{-}-\int_{\left[0,s\right[}\left[\upsilon-w+\nabla\psi\left(x\right)-\nabla\psi\left(y\right)\right]dt.
\]
Whence
\begin{equation}
\left\Vert \lambda^{-}\left(s\right)\left(x\left(s\right)-y\left(s\right)\right)+\left(\upsilon\left(s\right)-w\left(s\right)\right)\right\Vert \leq\int_{\left[0,s\right[}\left\Vert x-y\right\Vert \left|d\lambda^{-}\right|+\int_{\left[0,s\right[}\left\Vert \upsilon-w\right\Vert dt+\int_{\left[0,s\right[}\left\Vert \nabla\psi\left(x\right)-\nabla\psi\left(y\right)\right\Vert dt\label{52}
\end{equation}
Define $\theta\left(s\right)=\left(c_{0}^{2}\left\Vert x\left(s\right)-y\left(s\right)\right\Vert ^{2}+\left\Vert \upsilon\left(s\right)-w\left(s\right)\right\Vert ^{2}\right)^{1/2}$.
The same reasoning as in Theorem \ref{basic-stab} yields 
\[
\theta\left(s\right)\leq\left\Vert \lambda^{-}\left(s\right)\left(x\left(s\right)-y\left(s\right)\right)+\left(\upsilon\left(s\right)-w\left(s\right)\right)\right\Vert ,\:\forall s\in\left[0,T\right].
\]
$\vphantom{}$
Besides we also have $c_{0}\left\Vert x\left(s\right)-y\left(s\right)\right\Vert \leq\theta\left(s\right)$,   $\left\Vert \upsilon\left(s\right)-w\left(s\right)\right\Vert \leq\theta\left(s\right)$, and we have that $\nabla\psi$ is the gradient of a convex, continuously
differentiable function $\psi:H\longrightarrow\mathbb{R}$, and by  Lipschitz continuity property of $\nabla \psi$, and definition of $\theta$, we have 
\begin{align}\label{stab100}
\|\nabla \psi\left(x(t)\right)-\nabla \psi\left(y(t)\right)\| &\leq L_{\psi} \|x(t)-y(t) \| \nonumber  \\
& \leq \frac{L_{\psi}}{c_0} \theta (t).
\end{align}
Hence, with (\ref{52})
\begin{equation}
\theta\left(s\right)\leq\frac{1}{c_{0}}\int_{\left[0,s\right[}\theta\left|d\lambda^{-}\right|+\int_{\left[0,s\right[}\theta dt+\frac{1}{c_{0}}L_{\psi}\int_{\left[0,s\right[}\theta dt=\int_{\left[0,s\right[}\theta d\mu,\label{53}
\end{equation}
where $d\mu$ denotes the nonnegative measure $\frac{1}{c_{0}}\left|d\lambda^{-}\right|+\left(1+\frac{1}{c_{0}}L_{\psi}\right)dt$.

$\vphantom{}$

If $\theta\not\equiv0$ on $\left[0,T\right],$ define $t_{0}=\inf\left\{ t\in\left[0,T\right],\theta\left(t\right)>0\right\} $.
Note $t_{0}<T$ and $\theta\left(t_{0}\right)=0$, since $\theta$
is continuous. With (\ref{53}) we then have 
\begin{equation}
\theta\left(s\right)\leq\int_{\left]t_{0},s\right[}\theta d\mu,\; t_{0}<s\leq T.\label{54}
\end{equation}
In view of $\int_{\left]t_{0},t_{0}\right]} d\mu=0$ and of
the right continuity at $t_{0}$ of $t\longrightarrow\int_{\left]t_{0},t\right]} d\mu,$ \cite[Proposition 9.1 ]{Moreau}
there exists some $t_{1}\in\left]t_{0},T\right]$ such that $\int_{\left]t_{0},t_{1}\right]} d\mu<1/2.$
Let $M$ be an upper bound of $\theta$ on $\left[0,t_{1}\right]$;
from (\ref{54}) we deduce, for $s\in\left]t_{0},t_{1}\right]$
\[
\theta\left(s\right)\leq M\int_{\left]t_{0},s\right[} d\mu\leq M\int_{\left]t_{0},t_{1}\right[} d\mu\leq\frac{M}{2}.
\]
Hence $M/2$ is also an upper bound of $\theta$ on $\left[0,t_{1}\right],$
which necessarily entails $M=0$ and $\theta\equiv0$ on $\left[0,t_{1}\right]$.
But this is contradiction with the  definition of $t_{0}.$ Hence
$\theta\equiv0$ and $\left(x,\upsilon\right)\equiv\left(y,w\right)$
on $\left[0,T\right]$.
\end{proof}
Theorem 2.2 has a natural global version formulation:
\begin{corollary} \label{BV-corollary}
Suppose that $\lambda:\left[0,\infty\right[\rightarrow\left]0,\infty\right[$
is of bounded variation on $\left[0,T\right]$ and $\inf\lambda\left(\left[0,T\right]\right)>0$
for any $T<\infty.$ Let  $\upsilon_{0}\in\partial\varphi\left(x_{0}\right)$ and
$\upsilon_{0}\neq0$.  Then there is existence and uniqueness of a strong solution $\left(x,\upsilon\right):\left[0,\infty\right[\longrightarrow H\times H$
 of the Cauchy problem 
\begin{equation}
\lambda\dot{x}+\dot{\upsilon}+\upsilon+\nabla\psi\left(x\right)=0, \:\;\upsilon\left(t\right)\in \partial\varphi\left(x\left(t\right)\right)\;\; x\left(0\right)=x_{0},\;\upsilon\left(0\right)=\upsilon_{0}\label{29}
\end{equation}
where the first equality holds for almost all $t\in\left[0,\infty\right[,$
and the inclusion holds for all $t\in\left[0,\infty\right[.$
\end{corollary}
\begin{remark} { \rm
The results obtained in this paper still hold if $B=\nabla\psi$ the
gradient of a convex, continuously differentiable function $\psi:H\longrightarrow\mathbb{R}$
is replaced by a maximal monotone cocoercive operator $B:H\longrightarrow H$
(see \cite{AA} ).}
\end{remark}

\section{Perspective}
Let us list some interesting questions to be examined in the future:
\begin{enumerate}
	\item Study the asymptotic stability properties, corresponding to the case where $T=+\infty$, in connection with the convergence results of \cite{AAS} and \cite{AS}. 
	\item  Extend the results to the case where the Levenberg-Marquart regularization term is given in a closed-loop form, $\lambda(t)= \alpha (\|\dot{x} (t)\|)$  as in
\cite{ARS}.
\item Study the same questions for the corresponding forward-backward algorithms, see \cite{AA}.
	
\end{enumerate}

%\bibliographystyle{plain}
%\bibliography{xxx}

\begin{thebibliography}{10}


\bibitem{AA} B. Abbas, H. Attouch, Dynamical systems and forward-backward
algorithms associated with the sum of a convex subdifferential and
a monotone cocoercive operator, \textit{Optimization}, (2014) http://dx.doi.org/10.1080/02331934.2014.971412.

\bibitem{AAS} B. Abbas, H. Attouch, B. F. Svaiter, Newton-like dynamics
and forward-backward methods for structured monotone inclusions in
Hilbert spaces, \textit{J. Optim. Theory Appl.}, \textbf{161} (2014),
No. 2, pp. 331-360.


\bibitem{aabr} Alvarez, F.,  Attouch, H.,  Bolte, J.,   Redont, P.: 
A second-order gradient-like dissipative dynamical system with
  Hessian-driven damping. Application to optimization and mechanics.
  J. Math. Pures Appl.  \textbf{81}, 747--779 (2002).
   
 \bibitem{Ant} Antipin, A.S.:  Minimization of convex functions on convex sets by means of differential equations. 
 Differential Equations. \textbf{30}, 1365--1375  (1994). 
  
%\bibitem{ABR} H. Attouch, J. Bolte and P. Redont, \textit{Optimizing properties of an inertial dynamical system with geometric damping.
%Link with proximal methods},
%Control and Cybernetics, \textbf{31} (2002), pp. 643--657.
    
\bibitem{abc} Attouch, H.,  Briceno, L., Combettes,  P.L.:  A parallel
  splitting method for coupled monotone inclusions. SIAM J. Control Optim. \textbf{48}, 3246--3270 (2010).
  
%\bibitem{AR} H. Attouch and P. Redont, \textit{The second-order in time continuous Newton method},
%Approximation, optimization and mathematical economics (Pointe-\`a-Pitre, 1999), Physica, Heidelberg, 2001. 

\bibitem{ARS}  H. Attouch, P. Redont, B. F. Svaiter,
 Global convergence of a closed-loop regularized Newton method for solving monotone inclusions in Hilbert spaces, {\it J.  Optim. Theory  Appl.}, 157 (2013),
 No. 3, pp. 624--650. 

\bibitem{AS} Attouch, H.,  Svaiter, B.F.:  A continuous dynamical Newton-like approach to solving monotone inclusions.
SIAM J. Control Optim. \textbf{49}, 574--598  (2011).
    
%\bibitem{BB} J.-B. Baillon and H. Br\'ezis, \textit{Une remarque sur le
%  comportement asymptotique des semi-groupes non lin\'eaires}, Houston
 % J. Math.,  \textbf{2} (1976), pp. 5--7.
  
 \bibitem{BC} Bauschke, H.,   Combettes,  P.L.: Convex Analysis and Monotone Operator Theory. CMS books in Mathematics, Springer (2011).
  
  \bibitem{BT} Beck, A.,  Teboulle, M.:  Gradient-based algorithms with applications in signal recovery problems.
In: Palomar, D.,   Eldar,  Y. (eds.): Convex Optimization in Signal Processing and Communications, pp. 33--88. Cambridge University Press, (2010).

\bibitem{Bol} Bolte, J.:
Continuous gradient projection method in Hilbert spaces.  J.  Optim. Theory  Appl. \textbf{119}, 235--259  (2003).

\bibitem{Br} Br\'ezis, H.:   Op\'erateurs Maximaux Monotones et
  Semi-Groupes de Contractions dans les Espaces de Hilbert.
  North-Holland/Elsevier, New-York  (1973).
  
\bibitem{Br2}  Br\'ezis, H.: Analyse Fonctionnelle.
 Masson, Paris (1983).
  
%\bibitem{Brow} F. E. Browder, \textit{Existence and approximation of solutions of nonlinear variational inequalities},
%  Proc. Nat. Acad. Sci. USA, \textbf{56} (1966), 1080-1086. 
    
\bibitem{Bru} Bruck, R. E.: Asymptotic convergence of nonlinear
  contraction semigroups in Hilbert spaces. J. Funct. Anal.
  \textbf{18},  15-26   (1975).
   
%\bibitem{CNQ} X. Chen, Z. Nashed and L. Qi, \textit{Smoothing methods and semismooth methods for nondifferentiable operators equations},
% SIAM Journal on Numerical Analysis, \textbf{38} (2000), pp. 1200-1216.

\bibitem{CP}
Combettes, P.L.,  Pesquet,  J.-C.: Proximal splitting methods in signal processing. In: 
Bauschke, H.,   Burachik, R., Combettes, P.L., Elser, V., Luke, D.R.,  Wolkowicz, H. (eds.): Fixed-Point Algorithms for Inverse Problems in Science and Engineering, pp. 185-212. Springer, New York, (2011).

\bibitem{DS} Dennis, J. E.,   Schnabel, R.B.:  Numerical Methods for Unconstrained Minimization.
Prentice-Hall, Englewood Cliffs, NJ, 1983. Reprinted by SIAM publications (1993).
      
\bibitem{DB} Dunn, J.C., Bertsekas, D.P.: Efficient dynamic
  programming implementations of Newton's method for unconstrained
  optimal control problems.  J.  Optim. Theory  Appl. \textbf{63}, pp.  23--38 (1999).

\bibitem{FS} Ferreira, O.P., Svaiter, B.F.: Kantorovich's theorem on Newton's method on Riemannian manifolds. 
Journal of Complexity \textbf{18},  304--329 (2002).

\bibitem{Gu} G$\ddot{\mbox{u}}$ler, O.: On the convergence of the proximal point algorithm for convex minimization, 
    SIAM J. Control Optim. \textbf{29},  403--419 (1991).

\bibitem{HA}Haraux,  A.: Syst\`emes dynamiques dissipatifs et applications.
 RMA 17, Masson, Paris, (1991).    
       
\bibitem{ISJ} Iusem, A.N.,  Svaiter, B.F.,  Da Cruz Neto,  J.X.: Central paths, generalized proximal point methods,
and Cauchy trajectories in riemannian manifolds.
    SIAM J. Control Optim. \textbf{37},  566--588 (1989).
 
%\bibitem{JLT} J. Jacobsen, O. Lewis and B. Tennis, \textit{Approximations of
%  continuous Newton's method: An extension of Cayley's problem}, 
 %   Electronic J. of Differential Equations, \textbf{15}
 % (2007), pp.  163--173.
  
\bibitem{Mint}  Minty, G.J.: Monotone (nonlinear) operators in Hilbert spaces.
  Duke Mathematical Journal \textbf{29},  341--346 (1962).  
	
\bibitem{Moreau} Moreau, J.J.: Bounded variation in time. In: Moreau, J.J., Panagiotopoulos, P.D., Strang, G. (eds.)
Topics in Nonsmooth Mechanics, pp. 1–76. Birkhauser, Basel (1988)
	
  
\bibitem{NW} Nocedal, J.,   Wright, S.: Numerical Optimization.
  Springer Series in Operations Research, Springer-Verlag, New-York (1999).      
  
\bibitem{Op} Opial, Z.: Weak convergence of the sequence of successive approximations for nonexpansive mappings.
  Bull. Amer. Math. Soc. \textbf{73},   591--597  (1967).
  
\bibitem{PS} Peypouquet, J.,   Sorin, S.: Evolution equations for maximal monotone operators: asymptotic analysis
in continuous and discrete time,
  J. of Convex Analysis \textbf{17}, 1113-1163  (2010). 
    
% \bibitem{Ram} A. G. Ramm, \textit{Dynamical systems method for solving operator equations},
 % Mathematics in Science and Engineering, \textbf{208},
 %  Elsiever, 2006.
  
\bibitem{Ro} Rockafellar, R.T.: Monotone operators and  the proximal point algorithm. 
    SIAM J. Control Optim. \textbf{14},   877--898  (1976).
    
\bibitem{Ro2} Rockafellar, R.T.: Maximal monotone relations and  the second derivatives of nonsmooth functions. 
    Ann. Inst. Henri Poincar\'e \textbf{2}, 167--184 (1985).    
    
 \bibitem{Son} Sontag, E.D.: Mathematical Control Theory, second edition.  Springer-Verlag, New-York (1998).
     
\bibitem{SoSv} Solodov, M.V., Svaiter, B.F.: A globally convergent inexact Newton method for systems of monotone equations. In:
Fukushima, M.,  Qi, L., (eds): 
  Nonsmooth, Piecewise Smooth, Semismooth and Smoothing Methods,  pp. 355-369. Kluwer Academic Publishers, (1999).
 
\bibitem{Z} Zeidler, E.: Nonlinear Functional Analysis and its Applications, Part II: Monotone Operators. Springer-Verlag, New-York (1990).

\bibitem{ZM} Zhu, D.L., Marcotte,  P., Co-coercivity and its role in the convergence of iterative schemes
for solving variational inequalities. J. Optim. \textbf{6}, 714–-726   (1996).

\end{thebibliography}
\end{document}